\documentclass[12pt,a4paper]{amsart}

\usepackage{amssymb,xspace}
\usepackage{amstext}
\theoremstyle{plain}
\usepackage{amsbsy,amssymb,amsfonts,latexsym}
\usepackage[utf8]{inputenc}
\usepackage{xcolor}
\usepackage{graphicx} 
\usepackage{url}
\usepackage{a4wide}

\usepackage{enumerate}
\usepackage{graphics}

\newtheorem{theorem}{Theorem}[section]

\newtheorem{corollary}[theorem]{Corollary}
\newtheorem{proposition}[theorem]{Proposition}

\theoremstyle{definition}

\newtheorem{example}[theorem]{Example}

\theoremstyle{remark}

\theoremstyle{fact}
\newtheorem{fact}[theorem]{Fact}

\newenvironment{pf}
{\noindent\textbf{Proof:}}
{\hfill $\blacksquare$\par\medskip}
\newenvironment{rem}
{ \noindent {\it Remark\/}: }
{\null \hfill \par\medskip}

\begin{document}


\date{\today}

\title{Disjointly universal inner functions}

\author[\"O. Martin]{\"Ozgur Martin}

\author[D. Papathanasiou]{Dimitris Papathanasiou}

\author[S. \c{S}ahin]{Sibel \c{S}ahin}

\address{\"Ozgur Martin\\
Department of Mathematics, Mimar Sinan Fine Arts University, Istanbul, T\"urkiye}
\email{ozgurmartin@gmail.com}

\address{Dimitris Papathanasiou\\
Sabanci University Tuzla Campus, Orta Mahalle, \"Universite
Cadesi No:27 Tuzla, 34956 Istanbul, Turkey}
\email{d.papathanasiou@sabanciuniv.edu}

\address{Sibel \c{S}ahin\\
Department of Mathematics, Mimar Sinan Fine Arts University, Istanbul, T\"urkiye}
\email{sibel.sahin@msgsu.edu.tr}
\thanks{}



\begin{abstract} 
We characterize when two sequences of composition operators admit disjointly universal Blaschke products and singular inner functions. The characterizations we provide depend on geometric features of the symbols like their hyperbolic derivatives and pseudo hyperbolic distances. To achieve our results, we build a disjoint universality criterion for sequences of maps that act on a metrizable, complete topological semigroup.
\end{abstract}

\maketitle

\section{Introduction and preliminaries}

Let $B$ be a Blaschke product. If it is infinite, its zeros will accumulate in a subset of the unit circle $\mathbb{T}=\{z\in \mathbb{C}: |z|=1\}$, the closure of which is a singular set for $B$, in the sense that $B$ does not extend analytically to it. It is known that $B$ has a wild behavior close to those singular points, see \cite[Theorem 6.6]{ga}. But how wild can this behavior be? 

As usually, we denote $\mathbb{D}=\{z\in \mathbb{C}: |z|<1\}$ and for a subset $K\subset \mathbb{C}$ and a function $f$ bounded on $K$, $\|f\|_K=\sup_{z\in K}|f(z)|$. We also set 
$$
\mathcal{B}=\{f\in H(\mathbb{D}): \|f\|_{\mathbb{D}}\leq 1\}
$$ 
the closed unit ball of $H^{\infty}(\mathbb{D})$. Heins \cite{he} showed that if $(a_n)_n\subset \mathbb{D}$ is such that $a_n\rightarrow 1$, and if for $n\in \mathbb{N}$, $\phi_n=\frac{a_n-z}{1-\overline{a_n}z}$, then there is a Blaschke product satisfying that $\{B\circ \phi_n: n\in \mathbb{N}\}$ is dense in $\mathcal{B}$, when the latter is endowed with the topology of uniform convergence on compact subsets. It is evident, that $1$ must be a singular point for $B$. 

Heins's result can be expressed more succinctly if we introduce some terminology. If $\phi :\mathbb{D}\rightarrow \mathbb{D}$ is holomorphic, the {\it composition operator} 
$$
C_{\phi}(f)=f\circ \phi
$$
is a well-defined, continuous operator on $H(\mathbb{D})$ and $\mathcal{B}$ is a $C_{\phi}$ invariant subset. We also notice for future reference that $\mathcal{B}$ with the topology of uniform convergence on compact subsets, is a metrizable, complete, topological semigroup, under multiplication, and that $C_{\phi}$ is a homomorphism.  

Let $X$ be a topological space. The $N$ sequences of continuous self-maps of $X$,\\  $(T_{1,n})_n, (T_{2,n})_n,\dots ,(T_{N,n})_n$, $N\in \mathbb{N}$, are called {\it disjointly universal} if there is $x\in X$ such that the set
$$
\{(T_{1,n}(x),\dots ,T_{N,n}(x)): n\in \mathbb{N}\}
$$
is dense in $X^N$, when the latter is endowed with the product topology. Such an element $x$, provided that it exists, is called a {\it disjointly universal element} for the sequences $(T_{1,n})_n, (T_{2,n})_n,\dots ,(T_{N,n})_n$. When $N=1$, in which case we only consider one sequence of continuous self-maps, disjoint universality will be referred to as {\it universality}, and the disjointly universal elements as {\it universal elements}. 

Disjoint universality of sequences of maps is clearly stronger than the mere universality of each of them separately. Informally, disjointness requires some independence on the action of the sequences. In particular, it is immediate that if $X$ is not a singleton and it is Hausdorff, then no sequence of continuous maps on it is disjointly universal with itself. In this paper, when discussing disjointness, we will restrict ourselves to two sequences of maps. This serves to avoid a cumbersome notation and to keep the ideas as clear as possible since the passage from two to many will always require straightforward modifications. Only exception to this rule, will constitute the last section.

After this long detour, we notice that Heins's result expresses that the sequence $(C_{\phi_n})_n$ is universal on $\mathcal{B}$ and admits a universal Blaschke product. One could guess that the sequence of symbols $(\phi_n)$ consisting of disk automorphisms, simplifies to a certain extent the situation. The natural follow-up question, is to consider a sequence of holomorphic self-maps of $\mathbb{D}$, say $(\phi_n)_n$, such that $\phi_n(0)\rightarrow 1$, and to wonder when the sequence $(C_{\phi_n})_n$ is $\mathcal{B}$-universal and if this is the case, whether it admits a universal Blaschke product. Bayart et al. \cite{BGGM09} answered decisively this question, proving that if $(C_{\phi_n})_n$ is $\mathcal{B}$-universal, it admits a universal Blaschke product, and this happens precisely when $\limsup_{n\rightarrow \infty}\frac{|\phi_n'(0)|}{1-|\phi_n(0)|^2}=1$. An informal way to interpret this condition on hyperbolic derivatives, is to say that the symbols should be ``almost" disc automorphisms over a subsequence, see Fact \ref{fact}.

A closely related question which is addressed in \cite{BGGM09} is, assuming that $(C_{\phi_n})_n$ is universal, does it admit a universal singular inner function? The question is not well-posed as it is since, by Hurwitz's theorem, if $S$ is a singular inner function, any limit point of the set of $\{S\circ \phi_n : n\in \mathbb{N}\}$ (with the topology of uniform convergence on compact subsets) will either be a zero-free function, or the zero function. Let therefore 
$$
\mathcal{S}=\{f\in H^{\infty}(\mathbb{D}): 0<|f(z)|\leq 1, \forall z \in \mathbb{D}\}\cup \{0\}.
$$
We notice that $\mathcal{S}$ is a metrizable, complete topological subsemigroup of $\mathcal{B}$ which is invariant under $C_{\phi}$, if $\phi: \mathbb{D}\rightarrow \mathbb{D}$ is holomorphic. In \cite{BGGM09} it is shown that if $(C_{\phi_n})_n$ is $\mathcal{S}$-universal then it admits a universal singular inner function, which again happens precisely when $\limsup_{n\rightarrow \infty}\frac{|\phi_n'(0)|}{1-|\phi_n(0)|^2}=1$.

In this paper, we consider the disjoint analogues of the previous questions. Specifically, if $(\phi_{1,n})_n$ and $(\phi_{2,n})_n$ are two sequences of holomorphic self-maps of $\mathbb{D}$ satisfying that $\phi_{1,n}(0)\rightarrow a\in \mathbb{T}$, and $\phi_{2,n}(0)\rightarrow b\in \mathbb{T}$, we characterize when the sequences $(C_{\phi_{1,n}})_n$ and $(C_{\phi_{2,n}})_n$ admit a disjointly $\mathcal{B}$-universal Blaschke product or a disjointly $\mathcal{S}$-universal singular inner function. As in the universal case, the characterizing condition is common and has two components. The first component involves the hyperbolic derivatives of the symbols $\phi_{1,n}$, $\phi_{2,n}$, and it ensures that they behave almost like disc automorphisms. This is a consequence of the fact that disjoint universality implies universality. The second component, involves the pseudo-hyperbolic distance of the symbols, and it ensures that the images of compact subsets of $\mathbb{D}$ under $(\phi_{1,n})_n$ and $(\phi_{2,n})_n$ eventually get well separated. This keeps the action of $(C_{\phi_{1,n}})_n$ and $(C_{\phi_{2,n}})_n$ sufficiently independent. As an important corollary, we characterize when the sequences of iterates $(C_{\phi_1}^n)_n$ and $(C_{\phi_2}^n)_n$ admit a disjointly $\mathcal{B}$-universal Blaschke product, if $\phi_1$ and $\phi_2$ are holomorphic self-maps of $\mathbb{D}$. This corollary is reminiscent of the interesting results of Bayart on disjoint frequent hypercyclicity from \cite{ba}.

To prove our results we develop an abstract disjoint universality criterion similar to the universality criterion by Walmsley \cite{Dave}. The way we apply our criterion does not imitate \cite{BGGM09} or \cite{Dave}. Actually, we believe, that our proof, when restricted to one sequence of maps, simplifies the already existing proofs of the universality results from \cite{BGGM09} and \cite{Dave}. Loosely speaking, both ours and Walmsley's criteria, rely on constructing densely defined approximate right inverses for the original maps. The novelty of our approach stems from the fact that in order to define those approximate right inverses, one needs only use disc automorphisms, as opposed to the approaches from \cite{BGGM09} and \cite{Dave}, where in order to achieve exact right inverses, the use of thin Blaschke products was necessary.

Finally, inspired by the work of Grosse-Erdmann and Mortini \cite{GrMo} we prove that, as in the case of universality, disjoint universality of $(C_{\phi_{1,n}})_n$ and $(C_{\phi_{2,n}})_n$ on $\mathcal{B}$ implies their disjoint universality on $H(\mathbb{D})$. Informally, the reason for this, is that for disjoint universality on $H(\mathbb{D})$ the symbols must separate compact subsets of $\mathbb{D}$, as for $\mathcal{B}$, and must be ``eventually" univalent, not necessarily almost disc automorphisms. We conclude the paper with remarks and extensions of our results.

\section{Generalities}

We first establish the following analogue of Birkhoff's transitivity theorem (see \cite[Theorem 2.19]{gp} or \cite[Theorem 1.2]{bm}) for disjoint universality. 

\begin{proposition} \label{Birkhoff}
Let $X$ be a separable, metrizable, complete topological space, 
and for each $n\in \mathbb{N}$, let $T_{1,n}$ and $T_{2,n}$ be continuous self maps of $X$. The set of disjointly universal elements for the sequences $(T_{1,n})_n, (T_{2,n})_n$  is a $G_{\delta}$ subset of $X$, and it is residual if and only if, for each $U, V_1$, and $V_2$ non-empty open subsets of $X$, there exists $N\in \mathbb{N}$ such that
\begin{equation}\label{Baire}
U \cap T^{-1}_{1,N} (V_1) \cap T^{-1}_{2,N} (V_2) \neq \emptyset.
\end{equation}
\end{proposition}

\begin{pf}
    Since $X$ is separable and metrizable, it is second countable. Let $(V_k)_{k=1}^{\infty}$ be a base for the topology of $X$. The set of disjointly universal vectors for $(T_{1,n})_n, (T_{2,n})_n$  can be written as
    $$
    \bigcap_{i,j=1}^{\infty}\bigcup_{n=1}^{\infty}( T^{-1}_{1,n} (V_i) \cap T^{-1}_{2,n} (V_j))
    $$
    which is a $G_{\delta}$ subset of $X$. Assuming equation \ref{Baire}, the Baire category theorem proves the residuality of the set of disjointly universal elements for $(T_{1,n})_n$, $(T_{2,n})_n$. The reverse implication is trivial.
    \end{pf}

    We note that Sanders and Shkarin \cite{SaSh} provided examples of disjointly hypercyclic operators, whose sets of disjointly hypercyclic vectors are not dense. In the more general setting of universality, one may obtain simpler examples of two sequences of say bounded linear operators on a Hilbert space, which are disjointly universal but with a non-dense set of disjointly universal vectors.  

    The next criterion will be our main tool for proving disjoint universality, and more specifically, for singling out disjointly universal elements of a certain form. Its proof follows closely \cite[Theorem 2.1]{Dave} and, due to its length, it will be presented in the appendix. Let us remark that proving the first claim of Theorem \ref{criterion} is a simple application of Proposition \ref{Birkhoff}. The difficulty in proving the second claim, comes from the fact that disjointly universal elements of the form that interests us, need not be residual in the space. 

\begin{theorem} \label{criterion}
    Let $X$ be a separable, metrizable, complete topological semigroup with identity $e$, and for each $n\in \mathbb{N}$, let $T_{1,n}$ and $T_{2,n}$  be continuous homomorphisms on $X$. If there are dense subsets $D_0$, $D_1$, $D_2$ of $X$, a subsequence $(n_k)$ of $\mathbb{N}$, and maps $S_{1, n_k}: D_1\rightarrow X$ and $S_{2, n_k}:D_2\rightarrow X$ such that for $i,j\in \{1,2\}$,
    \begin{enumerate}
        \item $\lim_{k\rightarrow \infty}T_{i, n_k} x=e, x \in D_0$,
        \item $\lim_{k\rightarrow \infty}S_{i, n_k} x=e, x \in D_i$, and
        \item $\lim_{k\rightarrow \infty}T_{i, n_k}S_{j, n_k} x=\begin{cases}&x, \quad \text{if} \quad i=j,\\
        &e, \quad \text{if} \quad i\neq j, \end{cases}  \quad x \in D_j$,
    \end{enumerate}
    then the set of disjointly universal vectors for $(T_{1,n})_n, (T_{2,n})_n$ is a dense $G_{\delta}$ subset of $X$. Furthermore, there exists a disjointly universal vector for $(T_{1,n})_n, (T_{2,n})_n$ of the form
    $
    \prod_{i=1}^{\infty}x_i,
    $
    for $x_i \in D_0$.
\end{theorem}

   \begin{rem}
        We notice that when an $F$-space is considered as a topological group under addition, then the above theorem recovers the Disjoint Hypercyclicity Criterion \cite[Proposition 2.5]{BP07}. In the case of one sequence of maps in which disjoint universality becomes mere universality, the previous theorem gives \cite[Corollary 2.2]{Dave}.
    \end{rem}

\section{Disjointly universal Blaschke products}

    We now apply the results of the previous section for the case when $X=\mathcal{B}$, the closed unit ball of $H^{\infty}(\mathbb{D})$ endowed with the uniform convergence on compact subsets, to obtain a disjoint universality analogue of \cite[Theorem 2.1]{BGGM09}. Big part of applying Theorem \ref{criterion} consists of determining the densely defined approximate right inverses for $T_{1, n}$ and $T_{2, n}$. For composition operators that means finding approximate left inverses for the symbols. When the symbols are automorphisms this procedure is clear. In the non-automorphic case and when the domain is $\mathbb{D}$ one may use the Schwarz Lemma to achieve this. 
    
    We recall that  the \textit{pseudo-hyperbolic distance} of $z,w \in \mathbb{D}$ is defined by 
        $$
        \rho(z,w)=\left|\frac{z-w}{1-\overline{z}w}\right|.
        $$
    We may extend the definition of the pseudo-hyperbolic distance on the closed unit disc by setting $\rho(z,w)=1$, if $z\neq w$, and either $z$ or $w$ lie on $\mathbb{T}$. Of course, if $z=w\in \mathbb{T}$ we define $\rho(z,w)=0$. It is therefore clear, that $\rho(z,w)\in [0,1]$ and that $\rho(z,w)=1$ precisely when at least one of $z$ or $w$ belongs to $\mathbb{T}$ and $z\neq w$. It is also evident that $\rho$ is continuous on $(\overline{\mathbb{D}}\times \overline{\mathbb{D}})\setminus \{(z,z): z\in \mathbb{T}\}$. For a holomorphic function $f:\mathbb{D}\rightarrow \mathbb{D}$, its {\it hyperbolic derivative} at $0$, is 
    $$
    f^\#(0)=\frac{|f'(0)|}{1-|f(0)|^2}.
    $$
    Concerning the pseudo-hyperbolic distance and the hyperbolic derivative at zero, we will make repeated use of the following well-known properties whose proof relies on the Schwarz-Pick lemma. We will use the standard notation $Aut(\mathbb{D})$ for the group of automorphisms of $\mathbb{D}$.
    \begin{fact}\label{fact}
        Let $f,g\in \mathcal{B}$, then the following hold.
        \begin{enumerate}
            \item $\rho(f(z),f(w))\leq \rho(z,w), \forall z,w\in \mathbb{D}$.
            \item $f^\#(0)\in [0,1]$ and $f^\#(0)=1$ if and only if $f\in Aut(\mathbb{D})$.
            \item $f\mapsto f^\#(0):\mathcal{B}\rightarrow [0,1]$ is continuous.
            \item $(f\circ g)^\#(0)\leq g^\#(0)$.
        \end{enumerate}
    \end{fact}

    We are now ready to characterize disjoint universality for pairs of sequences $(C_{\phi_{1, n}})_n$, $(C_{\phi_{2, n}})_n$ whose symbols are converging to the boundary. We notice that the separation of $\phi_{1, n}(0)$ and $\phi_{2, n}(0)$ through the pseudo-hyperbolic distance at condition (3) is reasonable since, for disjoint universality to have a chance, there should be a condition ensuring that  $\phi_{1, n}\neq \phi_{2, n}$ eventually, at least over a subsequence.
    
    \begin{theorem}\label{Blaschke}
        Let $(\phi_{1, n})_{n=1}^{\infty}$ and $(\phi_{2, n})_{n=1}^{\infty}$ be sequences of holomorphic self maps of $\mathbb{D}$ and assume that $\phi_{1, n}(0)\rightarrow a$ and $\phi_{2, n}(0)\rightarrow b$, where $a, b\in \mathbb{T}$. The following assertions are equivalent.
        \begin{enumerate}
            \item The sequences $(C_{\phi_{1, n}})_n$, $(C_{\phi_{2, n}})_n$ admit a $\mathcal{B}$-universal function; 
            \item the sequences $(C_{\phi_{1, n}})_n$, $C_{\phi_{2, n}})_n$ admit a $\mathcal{B}$-universal Blaschke product;
            \item there is a subsequence $(n_k)$ such that
            $$
            \lim_{k\rightarrow \infty}\frac{|\phi_{1, n_k}'(0)|}{1-|\phi_{1, n_k}(0)|^2}=\lim_{k\rightarrow \infty}\frac{|\phi_{2, n_k}'(0)|}{1-|\phi_{2, n_k}(0)|^2}=1
            $$
            and
            $$
            \rho(\phi_{1, n_k}(0),\phi_{2, n_k}(0))\rightarrow 1.
            $$
            \end{enumerate}
    \end{theorem}

    \begin{rem}
        We observe that when $a\neq b$, $\rho(\phi_{1, n}(0),\phi_{2, n}(0))\rightarrow 1$, as $n\rightarrow \infty$, holds automatically. Hence condition (3) can be also be stated as
        \begin{enumerate}
            \item[($3$)] if $a\neq b$, there is a subsequence $(n_k)$ such that
            $$
            \lim_{k\rightarrow \infty}\frac{|\phi_{1, n_k}'(0)|}{1-|\phi_{1, n_k}(0)|^2}=\lim_{k\rightarrow \infty}\frac{|\phi_{2, n_k}'(0)|}{1-|\phi_{2, n_k}(0)|^2}=1.
            $$
            If $a=b$ it also holds that
            $$
            \rho(\phi_{1, n_k}(0),\phi_{2, n_k}(0))\rightarrow 1.
            $$
        \end{enumerate}
    \end{rem}

    \begin{pf}
        That $(2)\Rightarrow (1)$ is clear. We will show that $(1)\Rightarrow (3)$ and that $(3)\Rightarrow (2)$.

        $(1) \Rightarrow (3)$. Let $(a_k)\subset \mathbb{D}$ be such that $a_k\rightarrow a$. Set for $k\geq 1$,
        $$
        \tau_k(z)=\frac{z+a_k}{1+\overline{a_k}z}
        $$
        the non-Euclidean translation that maps $0$ to $a_k$. By Fact \ref{fact} (3), we get an increasing exhaustion $(K_k)$ of $\mathbb{D}$ by compact sets, and a sequence $(\delta_k)$ of positive numbers such that $\delta_k\rightarrow 0$, satisfying that
        $$
        |g^\#(0)-\tau_k^\#(0)|<\frac{1}{k}, \quad \text{if} \quad \|g-\tau_k\|_{K_k}<\delta_k.
        $$
        If $f$ is a disjointly universal vector for $(C_{\phi_{1, n}},C_{\phi_{2, n}})$ and if $i(z)=z$ is the identity function, there exists a subsequence $(n_k)$ such that
        $$
        \|f\circ \phi_{1, n_k}-\tau_k\|_{K_k}<\delta_k
        $$
        and
        $$
        \|f\circ \phi_{2, n_k}-i\|_{K_k}<\delta_k.
        $$
        Since $\tau_k\rightarrow a$ on compact subsets, we get that $f\circ \phi_{1, n_k}\rightarrow a$. Also, since $f\circ \phi_{2, n_k}\rightarrow i$, it follows that 
        $$
        (f\circ \phi_{2, n_k})^\#(0)\rightarrow i^\#(0)=1
        $$ 
        which, together with the fact that 
        $$
        (f\circ \phi_{2, n_k})^\#(0)\leq \phi_{2, n_k}^\#(0)\leq 1,
        $$ 
        yields that $\phi_{2, n_k}^\#(0)\rightarrow 1$.
        Concerning the pseudo-hyperbolic distances, we have that
        $$
        \rho((f\circ \phi_{1, n_k})(0),(f\circ \phi_{2, n_k})(0))\rightarrow \rho(a,0)=1
        $$
        and since
        $$
        \rho((f\circ \phi_{1, n_k})(0),(f\circ \phi_{2, n_k})(0))\leq \rho(\phi_{1, n_k}(0),\phi_{2, n_k}(0))\leq 1,
        $$
        we conclude that 
        $$
        \rho(\phi_{1, n_k}(0),\phi_{2, n_k}(0))\rightarrow 1.
        $$
        Finally, we have that
        $$
        |(f\circ \phi_{1, n_k})^\#(0)-1|=|(f\circ \phi_{1, n_k})^\#(0)-\tau_k^\#(0)|<\frac{1}{k}
        $$
        hence, 
        $$
        (f\circ \phi_{1, n_k})^\#(0)\rightarrow 1.
        $$
        Since
        $$
        (f\circ \phi_{1, n_k})^\#(0)\leq \phi_{1, n_k}^\#(0)\leq 1
        $$
        it follows that $\phi_{1, n_k}^\#(0)\rightarrow 1$.

        For $(3)\Rightarrow (2)$ we apply Theorem \ref{criterion}. Using the fact that for each $T^{(i)}_n=C_{\phi^{(i)}_n}, i=1,2$ we can define a corresponding $S^{(i)}_n$, we can simplify the argument from \cite[Theorem 2.1]{BGGM09} and avoid using thin sequences and thin Blaschke products. For each $n\in \mathbb{N}$, we set
        $$
        u_{1, n}(z)=\frac{z-\phi_{1, n}(0)}{1-\overline{\phi_{1, n}(0)}z},\quad z\in \mathbb{D}
        $$
        and 
        $$
         u_{2, n}(z)=\frac{z-\phi_{2, n}(0)}{1-\overline{\phi_{2, n}(0)}z},\quad z\in \mathbb{D}.
        $$
        It is clear that
         $$
        u_{1, n}\rightarrow -a
        $$
        and that
        $$
        u_{2, n}\rightarrow -b.
        $$
        Since
        $$
        \rho(\phi_{1,n_k}(0),\phi_{2,n_k}(0))=|u_{1, n_k}\circ \phi_{2, n_k}(0)|=|u_{2, n_k}\circ \phi_{1, n_k}(0)|\rightarrow 1,
        $$
        by passing to a subsequence if necessary, we may assume that for some $c, d\in \mathbb{T}$,
        $$
        u_{1, n_k}\circ \phi_{2, n_k}\rightarrow c,
        $$
        and
        $$
        u_{2, n_k}\circ \phi_{1, n_k}\rightarrow d.
        $$
        
        By a normal families argument, and by passing to a further subsequence if necessary, we may assume that
        $$
        u_{1, n_k}\circ \phi_{1, n_k}\rightarrow f,
        $$
        for some $f\in \mathcal{B}$. Since $(u_{1, n_k}\circ \phi_{1, n_k})(0)=0$, for each $k\in \mathbb{N}$, it holds that $f(0)=0$. Since
        $$
        |f'(0)|=\lim_{k\rightarrow \infty}\frac{|\phi_{1, n_k}'(0)|}{1-|\phi_{1, n_k}(0)|^2}=1,
        $$
        Schwarz's Lemma implies that $f(z)=\lambda z$, for some $\lambda \in \mathbb{T}$, that is
        $$
        (u_{1, n_k}\circ \phi_{1, n_k})(z)\rightarrow \lambda z
        $$
        uniformly on compact subsets.
        By passing to a further subsequence if necessary, and by repeating the above argument we get a $\mu \in \mathbb{T}$ such that
        $$
        (u_{2, n_k}\circ \phi_{2, n_k})(z)\rightarrow \mu z
        $$
        uniformly on compact subsets.
        We now consider the sets
        $$
        D_0=\{B \quad \text{finite Blaschke product}: B(a)=B(b)=1\},
        $$
        $$
        D_1=\{B \quad \text{finite Blaschke product}: B(-a\overline{\lambda})=B(c\overline{\lambda})=1\},
        $$
        and
        $$
        D_2=\{B \quad \text{finite Blaschke product}: B(-b\overline{\mu})=B(d\overline{\mu})=1\}.
        $$
        By \cite[Lemma 3.9]{Dave} we know that $D_0, D_1$ and $D_2$ are dense in $\mathcal{B}$. 
        We finally define $S_{1, n}: D_1\rightarrow \mathcal{B}$ by
        $$
        S_{1, n}=C_{\overline{\lambda}u_{1, n}}
        $$
        and $S_{2, n}: D_2\rightarrow \mathcal{B}$ by
        $$
        S_{2, n}=C_{\overline{\mu}u_{2, n}}.
        $$
        It is clear that 
        $$
        T_{1,n_k}(B)\rightarrow 1
        $$ 
        and that 
        $$
        T_{2,n_k}(B)\rightarrow 1$$ 
        in $\mathcal{B}$, for each $B\in D_0$, since $C_{\phi_{1, n_k}}(B)(0)=B(\phi_{1, n_k}(0))\rightarrow B(a)=1$ and $C_{\phi_{2, n_k}}(B)(0)=B(\phi_{2, n_k}(0))\rightarrow B(b)=1$. 
       We also have that 
       $$
       S_{1, n_k}(B)\rightarrow 1,
       $$ 
       for $B\in D_1$, since $C_{\overline{\lambda}u_{1, n_k}}(B)(0)=B(\overline{\lambda}u_{1, n_k}(0))\rightarrow B(-a\overline{\lambda})=1$. Similarly, we get that 
       $$
       S_{2, n_k}(B)\rightarrow 1,
       $$ 
       for $B\in D_2$. 
        It also holds that 
        $$
T_{1, n_k}S_{1, n_k}(B)\rightarrow B,
        $$ 
        for every $B\in D_1$, since $C_{\phi_{1, n_k}}C_{\overline{\lambda}u_{1, n_k}}(B)(z)=B(\overline{\lambda}u_{1, n_k}\circ \phi_{1, n_k})(z)\rightarrow B(\overline{\lambda}\lambda z)= B(z)$ and since pointwise convergence implies uniform convergence on compact subsets on $\mathcal{B}$. Similarly, we get that
        $$
T_{2, n_k}S_{2, n_k}(B)\rightarrow B,
        $$
        for every $B\in D_2$. Finally, we have that
        $$
T_{1, n_k}S_{2, n_k}(B)\rightarrow 1, 
        $$
        for $B\in D_2$, since $C_{\phi_{1, n_k}}C_{\overline{\mu}u_{2, n_k}}(B)(z)=B(\overline{\mu}u_{2, n_k}\circ \phi_{1, n_k})(z)\rightarrow B(d\overline{\mu})=1$. Similarly,
        $$
T_{2, n_k}S_{1, n_k}(B)\rightarrow 1,
        $$
        for $B\in D_1$. Theorem \ref{criterion} now provides the existence of a disjointly $\mathcal{B}$- universal vector for $(C_{\phi_{1, n}},C_{\phi_{2, n}})$ of the form 
        $$
        \prod_{i=1}^{\infty}B_j, \quad B_j\in D_1
        $$
        which is a Blaschke product and hence way may conclude.
        \end{pf}

    \begin{rem}
        If we wanted to align Theorem \ref{criterion} to \cite[Theorem 2.1]{BGGM09}, we could have added the following equivalent condition.
        \begin{enumerate}
            \item[(4)] For each $m\geq 1$, the set $\mathcal{B}_m=\{(u\circ\phi_{1, n},u\circ \phi_{2, n}): u\in \mathcal{B}, n\geq m\}$ is dense in $\mathcal{B}$.
        \end{enumerate}
        That $(1)\Rightarrow(4)$ is clear since for a disjointly universal vector $f$ for $(C_{\phi_{1, n}},C_{\phi_{2, n}})$, and for any $m\geq 1$, we have that the set $\{(f\circ \phi_{1, n},f\circ \phi_{2, n}): n\geq m\}$ is dense in $\mathcal{B}$ and contained in $\mathcal{B}_m$. To see that $(4)\Rightarrow (3)$ we need to repeat step by step the proof of $(1)\Rightarrow (3)$, substituting the disjointly universal vector $f$ by $u_k \in \mathcal{B}$ satisfying that $\|u_k\circ \phi_{1, n_k}-\tau_k\|_{K_k}<\delta_k$, for a subsequence $(n_k)$, and that $\|u_k\circ \phi_{2, n_k}-i\|_{K_k}<\delta_k$. We did not include this condition in the theorem since we do not use it in the paper.
    \end{rem}


    Our first corollary refers to disjoint universality in the case that the two sequences of operators are sequences of iterates of two operators, or equivalently in the case of composition operators, when the sequences of symbols are iterates of two self-maps of the unit disc. To simplify the notation, for the following result we will denote the $n$-th iterate of a self-map of the unit disc $\phi$ by $\phi^n$ when $n\geq 1$. We will set $\phi^0(z)=z$, and in case $\phi$ is invertible, we will as usually denote $\phi^{-n}=(\phi^{-1})^n$, for $n\geq 1$. If $\phi$ is a holomorphic self-map of the unit disc, it is plain to see that the iterates of the composition operator with symbol $\phi$ are related to the iterates of its symbol through
    $$
    C_{\phi}^n=C_{\phi^n}, n\geq 0.
    $$
    The following is the disjoint analogue of \cite[Corollary 4.1]{BGGM09}. For the classification of disc automorphisms and their dynamics, we refer to \cite{Sh}.

    \begin{corollary}
        Let $\phi_1$ and $\phi_2$ be holomorphic self maps of $\mathbb{D}$. The sequences $(C_{\phi_1^n})_n$, $(C_{\phi_2^n})_n$ admit a disjointly universal Blaschke product if and only if $\phi_1\neq \phi_2$, both $\phi_1$ and $\phi_2$ are either parabolic or hyperbolic automorphisms of the disc (not necessarily of the same type), and in case $\phi_1$ and $\phi_2$ are both of hyperbolic type and they have the same attractive point, their multipliers are different.  
    \end{corollary}

    \begin{pf}
        The necessity of $\phi_1\neq \phi_2$ follows immediately from the fact that a sequence of maps is never disjointly universal with itself, while the necessity of both $\phi_1$ and $\phi_2$ being parabolic or hyperbolic automorphisms follows from \cite[Corollary 4.1]{BGGM09} and the observation that disjoint universality of a pair of sequences of maps, implies universality for each one of them. We also notice that since the iterates of disc automorphisms are also disc automorphisms, we have that $(\phi_1^n)^\#(0)=(\phi_2^n)^\#(0)=1$, for every $n\geq 0$. 

        Assuming that $\phi_1$ and $\phi_2$ are disc automorphisms of either parabolic or hyperbolic type, there are $a,b\in \mathbb{T}$ (the Denjoy-Wolff points) such that
        $$
        \phi_1^n(0)\rightarrow a,
        $$
        and
        $$
        \phi_2^n(0)\rightarrow b.
        $$
        In order to finalize the proof, Theorem \ref{criterion} suggests that we need to characterize when $\rho(\phi_1^n(0),\phi_2^n(0))\rightarrow 1$, at least over a subsequence.
        
        If $a\neq b$ Theorem \ref{criterion} applies right away for the whole sequence $(n)$ and we are done. If $a=b$ we notice that
        $$
        \rho(\phi_1^n(0),\phi_2^n(0))=|(\phi_1^{-n}\circ \phi_2^n)(0)|.
        $$
        If $\phi_1$ and $\phi_2$ commute, which happens exactly when they have the same fixed points, we get that
        $$
        \rho(\phi_1^n(0),\phi_2^n(0))=|(\phi_1^{-1}\circ \phi_2)^n(0)|.
        $$
        Since $\phi_1\neq \phi _2$, it holds that $\phi_1^{-1}\circ \phi_2$ is not the identity map. Furthermore, since $a$ is a common fixed point for $\phi_1$ and $\phi_2$, it is also a fixed point for $\phi_1^{-1}\circ \phi_2$. That means that the latter is an automorphism of the disc of either parabolic or hyperbolic type and hence
        $$
        |(\phi_1^{-1}\circ \phi_2)^n(0)|\rightarrow 1.
        $$
        Now Theorem \ref{criterion} applies for the whole sequence $(n)$ and concludes the proof. 

        If $\phi_1$ and $\phi_2$ do not commute, then since on one hand they do not have the same fixed points, and on the other hand they share the attractive fixed point $a\in \mathbb{T}$, we have two cases which we consider separately.

        Let us first assume that $\phi_1$ is of hyperbolic type with attractive fixed point at $a$, and $\phi_2$ is of parabolic type having $a$ as its unique fixed point. Let $\tau$ be the linear fractional transformation mapping $\mathbb{D}$ to the upper half plane $\mathbb{H}=\{z\in \mathbb{C}: \text{Im}(z)>0\}$ and sending $a$ to $\infty$ and the repelling point of $\phi_1$ to $0$. Then, there are $\lambda >1$ and $c\in \mathbb{R}$, such that if
        $$
        f_1(z)=\lambda z
        $$
        and 
        $$
        f_2(z)=z+c,
        $$
        $\phi_1=\tau^{-1}\circ f_1\circ \tau$ and $\phi_2=\tau^{-1}\circ f_2\circ \tau$. Since for $n\in \mathbb{N}$ and $z\in \mathbb{H}$,
        $$
        f_1^{-n}\circ f_2^n(z)=\frac{z+nc}{\lambda ^n}\rightarrow 0,
        $$
        we get that 
        $$
        \rho(\phi_1^n(0),\phi_2^n(0))=|(\tau^{-1}\circ f_1^{-n}\circ f_2^n \circ \tau)(0)|\rightarrow 1.
        $$

        If $\phi_1$ and $\phi_2$ are both of hyperbolic type, with the same attractive point $a$ and different repelling points, then if $\psi$ is the linear fractional transformation that maps $\mathbb{D}$ to $\mathbb{H}$, $a$ to $\infty$ and the repelling point of $\phi_1$ to $0$, there are $\gamma_1, \gamma_2>1$ and $x_0\in \mathbb{R}\setminus \{0\}$, such that if
        $$
        f_1(z)=\gamma_1 z
        $$
        and 
        $$
        f_2(z)=\gamma_2(z-x_0)+x_0,
        $$
        $\phi_1=\psi^{-1}\circ f_1\circ \psi$, and $\phi_2=\psi^{-1}\circ f_2\circ \psi$. A simple calculation shows that if $n\in \mathbb{N}$ and $z\in \mathbb{H}$, then
        $$
        f_1^{-n}\circ f_2^n(z)=\gamma_1^{-n}\gamma_2^n(z-x_0)+\gamma_1^{-n}x_0\rightarrow 
        \begin{cases}
            &0 \quad \text{if} \quad \gamma_1>\gamma_2,\\
            &\infty, \quad \text{if} \quad \gamma_1<\gamma_2,\\
            &z-x_0, \quad \text{if} \quad \gamma_1=\gamma_2.
        \end{cases}
        $$
        That yields that 
        $$
        \rho(\phi_1^n(0),\phi_2^n(0))=|(\psi^{-1}\circ f_1^{-n}\circ f_2^n \circ \psi)(0)|\rightarrow 1
        $$
        if and only if $\gamma_1 \neq \gamma_2$ which means if and only if the multipliers of $\phi_1$ and $\phi_2$ are different.
    \end{pf}

    Next, we would like to give some concrete applications of Theorem \ref{criterion} for the case that the symbols are not automorphisms of the disc. The first example is in the spirit of \cite[Example 4.6]{BGGM09} given a quantified flavor. 

    \begin{example}
      Let $(r_n)\subset (0,1)$ such that $r_n\rightarrow 1$, and $(a_n)\subset \mathbb{D}$ such that $a_n\rightarrow a\in \mathbb{T}$. For $n\in \mathbb{N}$, we let
      $$
      \phi_n(z)=\frac{r_nz+a_n}{1+r_n\overline{a_n}z}
      $$
      which is the composition of a dilation followed by an automorphism of the disc. Clearly, $\phi_n$ is not surjective. A direct calculation shows that
      $$
      |\phi_n'(0)|=r_n(1-|a_n|^2)=r_n(1-|\phi_n(0)|^2)
      $$
      hence, $\phi_n^\#(0)=r_n\rightarrow 1$. By \cite[Theorem 2.1]{BGGM09}, $(C_{\phi_n})$ has a universal Blaschke product. Keeping the same sequence $(r_n)$ and changing the sequence $(a_n)$ so that it converges to a different point of the unit circle, we get another sequence of self-maps of the disc, say $\psi_n$. Now Theorem \ref{criterion} applies, and we get examples of disjointly universal Blaschke products for the sequence $(C_{\phi_n},C_{\psi_n})$.
    \end{example}

    In the previous example, the sequence of symbols was still injective. In the following example we violate injectivity. The idea should be clear by now. Condition (3) of Theorem \ref{criterion} forces the symbols to become ``almost" automorphisms of the disc on some exhaustion of $\mathbb{D}$ with compact sets (over a subsequence). However, one still gets a diminishing amount of space to violate injectivity. We are again addressing first universality and then the transition to disjoint universality becomes easy.

    \begin{example}
     We first consider a sequence of self-maps of the disc $(\phi_n)$ such that $\phi_n(0)\rightarrow a \in \mathbb{T}$ and satisfying that $\phi_n^\#(0)\rightarrow 1$. Let $(K_n)$ be an increasing exhaustion of $\mathbb{D}$ with compact sets, and $(\delta_n)$ a sequence of positive numbers such that $\delta_n \rightarrow 0$ and satisfying that for $f\in \mathcal{B}$,
     $$
     \|f-\phi_n\|_{K_n}<\delta_n \Rightarrow |f^\#(0)-\phi_n^\#(0)|<\frac{1}{n}.
     $$
     We notice that the above is possible using the continuity of the map $h\mapsto h^\#(0)$ at $\phi_n$. We want, for each $n\geq 1$, to pick a non-injective function $f_n\in \mathcal{B}$ satisfying that $\|f_n-\phi_n\|_{K_n}<\delta_n$. One way to achieve this is the following. We fix $z_1\neq z_2$ on $\mathbb{T}$ and for each $n\geq 1$, we pick a finite Blaschke product $B_n$ satisfying that $B_n(z_1)=B_n(z_2)$ and that $\|B_n-\phi_n\|_{K_n}<\delta_n$. By substituting $B_n$ by $\|B_n\|^{-1}_{r\overline{\mathbb{D}}}B_n(rz)$, $z\in \mathbb{D}$ for some $r>1$ sufficiently close to $1$, we may assume that $B_n$ is not injective in $\mathbb{D}$ (it does not need to be a Blaschke product anymore). We now have that $B_n(0)\rightarrow a$ and that 
     $$
     |B_n^\#(0)-1|\leq |B_n^\#(0)-\phi_n^\#(0)|+|\phi_n^\#(0)-1|\rightarrow 0
     $$
     hence, the sequence $C_{B_n}$ admits a universal Blaschke product.

     Finally, modifying the initial sequence $\phi_n$ so that it converges to a different point of the unit circle we end up with a sequence $(C_{B_n},C_{\tilde{B}_n})$ with non-injective symbols, that admits a disjointly universal Blaschke product.
    \end{example}

    \section{Disjointly universal singular inner functions}

    Let $(\phi_{1, n})$ and $(\phi_{2, n})$ be two sequences of self maps of $\mathbb{D}$ with $\phi_{1, n}(0)\rightarrow a$ and $\phi_{2, n}(0)\rightarrow b$, for $a,b \in \mathbb{T}$. We want to determine if and when the sequences of operators $(C_{\phi_{1, n}},C_{\phi_{2, n}})$ admit a disjointly universal singular inner function. We first notice that since by Hurwitz's theorem, a sequence of zero free functions which converges on compact sets, has as a limit function either a zero free or the constantly zero function, the problem makes sense if we consider the composition operators $C_{\phi_{1, n}}$, and $C_{\phi_{2, n}}$ acting on the topological semigroup 
    $$
    \mathcal{S}:=\{f\in H(\mathbb{D}): 0<|f(z)|\leq 1, \forall z\in \mathbb{D}\} \cup \{0\}.
    $$
    It turns out that the existence of a disjointly $\mathcal{B}$-universal Blaschke product is equivalent to the existence of a disjointly $\mathcal{S}$-universal singular inner function for the pair of sequences $(C_{\phi_{1, n}},C_{\phi_{2, n}})$. In the following proof, we make an extensive use of the fact that functions from $\mathcal{S}$ have analytic roots. 

     \begin{theorem}\label{singular}
        Let $(\phi_{1, n})_{n=1}^{\infty}$ and $(\phi_{2, n})_{n=1}^{\infty}$ be sequences of holomorphic self maps of $\mathbb{D}$ and assume that $\phi_{1, n}(0)\rightarrow a$ and $\phi_{2, n}(0)\rightarrow b$, where $a, b\in \mathbb{T}$. The following assertions are equivalent.
        \begin{enumerate}
            \item The sequence $(C_{\phi_{1, n}},C_{\phi_{2, n}})$ admits an $\mathcal{S}$-universal singular inner function;
            \item there is a subsequence $(n_k)$ such that
            $$
            \lim_{k\rightarrow \infty}\frac{|\phi_{1, n_k}'(0)|}{1-|\phi_{1, n_k}(0)|^2}=\lim_{k\rightarrow \infty}\frac{|\phi_{2, n_k}'(0)|}{1-|\phi_{2, n_k}(0)|^2}=1
            $$
            and
            $$
            \rho(\phi_{1, n_k}(0),\phi_{2, n_k}(0))\rightarrow 1.
            $$
            \end{enumerate}
    \end{theorem}

    \begin{pf}
        $(1)\Rightarrow (2)$. We consider the atomic singular inner function
        $$
        S(z)=e^{\frac{z+1}{z-1}}, z\in \mathbb{D}.
        $$
        For each $k\in \mathbb{N}$, we denote by $S^{1/k}$ the analytic branch of the $k$-th root of $S$, satisfying that $S^{1/k}(0)>0$. Concretely, 
        $$
        S^{1/k}(z)=e^{\frac{z+1}{k(z-1)}}, z\in \mathbb{D}.
        $$
        Let $(K_k)$ be an exhaustion of $\mathbb{D}$ with compact sets. Pick inductively a sequence of positive numbers $(\delta_k)$ such that $\delta_k\rightarrow 0$ and satisfying that
        $$
        \|g-S\|_{K_k}<\delta_k \Rightarrow \|g^k-S^k\|_{K_k}<\frac{1}{k}, \quad g\in \mathcal{S},
        $$
        and
        $$
        \|g-S^{1/k^2}\|_{K_k}<\delta_k \Rightarrow \max\{ \|g^k-S^{1/k}\|_{K_k}, \|g^{k^2}-S\|_{K_k}\}<\frac{1}{k}, \quad g\in \mathcal{S}.
        $$
        Let $f\in \mathcal{S}$ be a disjointly universal function for $(C_{\phi_{1, n}},C_{\phi_{2, n}})$. Find a subsequence $(n_k)$ such that
        $$
        \|f\circ \phi_{1, n_k}-S\|_{K_k}<\delta_k
        $$
        and
        $$
        \|f\circ \phi_{2, n_k}-S^{1/k^2}\|_{K_k}<\delta_k.
        $$
        It follows that $(f^k\circ \phi_{1, n_k})(0)=(f\circ \phi_{1, n_k})^k(0)\rightarrow 0$, since $S^{k}(0)=1/e^{k}\rightarrow 0$. Similarly, $(f^k\circ \phi_{2, n_k})(0)=(f\circ \phi_{2, n_k})^k(0)\rightarrow 1$, since $S^{1/k}(0)=1/e^{1/k}\rightarrow 1$. We now have
        $$
        \rho((f^k\circ \phi_{1, n_k})(0),(f^k\circ \phi_{2, n_k})(0))\leq  \rho(\phi_{1, n_k}(0),\phi_{2, n_k}(0))\leq 1
        $$
        and since $ \rho((f^k\circ \phi_{1, n_k})(0),(f^k\circ \phi_{2, n_k})(0))\rightarrow 1$ we get that $\rho(\phi_{1, n_k}(0),\phi_{2, n_k}(0))\rightarrow 1$. It clearly also holds that
        $$
        f\circ \phi_{1, n_k}\rightarrow S
        $$
        and that
        $$
        f^{k^2}\circ \phi_{2, n_k}=(f\circ \phi_{2, n_k})^{k^2}\rightarrow S,
        $$
        uniformly on compact sets. Since $|S'(0)|=2/e$, it follows that
        $$
        |(f\circ \phi_{1, n_k})'(0)|=|f'((\phi_{1, n_k})(0))|(1-|\phi_{1, n_k}(0)|^2)\phi_{1, n_k}^\#(0)\rightarrow \frac{2}{e}
        $$
        and
        $$
        |(f^{k^2}\circ \phi_{2, n_k})'(0)|=|(f^{k^2})'((\phi_{2, n_k})(0))|(1-|\phi_{2, n_k}(0)|^2)\phi_{2, n_k}^\#(0)\rightarrow \frac{2}{e}.
        $$
        For $F\in \mathcal{S}$ and $\alpha \in \mathbb{D}$, if 
        $$
        \tau_{\alpha}(z)=\frac{\alpha-z}{1-\overline{\alpha}z}, z\in \mathbb{D},
        $$
        we have from \cite[Lemma 2.6]{BGGM09} that $|(F\circ \tau_{\alpha})'(0)|\leq 2/e$. Therefore,
        $$
        |(F\circ \tau_{\alpha})'(0)|=|F'(\alpha)|(1-|\alpha|^2)\leq \frac{2}{e}.        
        $$
        Setting $F=f$ and $\alpha=\phi_{1, n_k}(0)$, we get that
        $$
        |f'(\phi_{1, n_k}(0))|(1-|\phi_{1, n_k}(0)|^2)\leq \frac{2}{e}
        $$
        and since $\phi_{1, n_k}^\#(0)\leq 1$, we conclude that $\phi_{1, n_k}^\#(0)\rightarrow 1$. Repeating the above argument with $F=f^{k^2}$ and $\alpha=\phi_{2, n_k}(0)$, we get that $\phi_{2, n_k}^\#(0)\rightarrow 1$.

        For the implication $(2)\Rightarrow (1)$ one applies Proposition \ref{criterion} with $X=\mathcal{S}$ and follow varbatim the proof of implication $(3)\Rightarrow (2)$ from Theorem \ref{Blaschke}. The only necessary modifications are that now we define the sets  
        $$
        D_1=\{S \quad \text{singular inner function}: S(a)=S(b)=1\},
        $$
        $$
        D_2=\{S \quad \text{singular inner function}: S(-a\overline{\lambda})=S(c\overline{\lambda})=1\}
        $$
        and
        $$
        D_3=\{S \quad \text{singular inner function}: S(-b\overline{\mu})=S(d\overline{\mu})=1\},
        $$
        which by \cite[Lemma 3.15]{Dave}, are dense in $\mathcal{S}$. The result is a disjointly universal function which is the infinite product of singular inner functions and hence, itself a singular inner function. 
    \end{pf}

    \section{Disjoint universality in $H(\mathbb{D})$}
    In this section we are comparing the properties of disjoint $\mathcal{B}$-universality with disjoint $H(\mathbb{D})$-universality. Concerning normal universality, Grosse-Erdmann and Mortini \cite{GrMo} proved that $\mathcal{B}$-universality implies $H(\mathbb{D})$-universality. We could have adapted their argument in our setting and establish the analogue statement. However, we chose to compare the characterizing properties of the two disjoint universalities which leads to a more controlled understanding of why one property is stronger than the other. 

    Let us start our discussion with a sequence $(\phi_n)$ of holomorphic self maps of $\mathbb{D}$. It is clear that $|\phi_n(0)|\rightarrow 1$ if and only if, for each $K\subset \mathbb{D}$ compact, there is $n_0\in \mathbb{N}$ such that $\phi_n(K)\cap K=\emptyset$, for $n\geq n_0$.

    Assume now, that the sequence $(\phi_n)$ satisfies that $\phi_n^\#(0)\rightarrow 1$. It follows then that there is a subsequence $(n_k)$ such that for each $K\subset \mathbb{D}$ compact, there is $k_0\in \mathbb{N}$, satisfying that $\phi_{n_k}|_K$ is univalent for each $k\geq k_0$. Indeed, let 
    $$
    u_n(z)=\frac{z-\phi_n(0)}{1-\overline{\phi_n(0)}z}.
    $$
    Using the Schwarz Lemma and a normal families argument as we did in the proof of Theorem \ref{Blaschke}, we get the existence of a subsequence $(n_k)$ and a $\lambda \in \mathbb{T}$ such that 
    $$
    (u_{n_k}\circ \phi_{n_k})(z)\rightarrow \lambda z
    $$ 
    uniformly on compact subsets of $\mathbb{D}$. From \cite[Lemma 5.1]{BGGM09} we get that the functions $(u_{n_k}\circ \phi_{n_k})|_K$ are eventually univalent which yields the same conclusion for the functions $\phi_{n_k}|_K$. The reverse is of course not true as can be seen by considering 
    $$
    \phi_n(z)=a_nz+1-a_n,
    $$
    where $0<a_n<1$, and $a_n\rightarrow 0$.

    We next consider two sequences of self maps of $\mathbb{D}$, say $(\phi_{1,n})$ and $(\phi_{2,n})$. We claim that the assumption $\rho(\phi_{1,n}(0),\phi_{2,n}(0))\rightarrow 1$ implies that for each $K\subset \mathbb{D}$ compact, there is $n_0\in \mathbb{N}$ such that 
    $$
    \phi_{1,n}(K)\cap \phi_{2,n}(K)=\emptyset, \quad n\geq n_0.
    $$
    Indeed, if 
    $$
     u_{1,n}(z)=\frac{z-\phi_{1,n}(0)}{1-\overline{\phi_{1,n}(0)}z}
    $$
    then the Schwarz Lemma implies that $(u_{1,n}\circ \phi_{1,n})(K)\subset K$. The assumption is that $|(u_{1,n}\circ \phi_{2,n})(0)|\rightarrow 1$ which, as discussed above, implies that $(u_{1,n}\circ \phi_{2,n})(K)\cap K=\emptyset$, eventually. This gives that $(u_{1,n}\circ \phi_{2,n})(K)\cap (u_{1,n}\circ \phi_{1,n})(K)=\emptyset$ and hence that $\phi_{1,n}(K)\cap \phi_{2,n}(K)=\emptyset$ eventually. Clearly the reverse is not true if we consider the sequences of constant functions $\phi_{1,n}=0$, and $\phi_{2,n}=1/2$, for each $n\in \mathbb{N}$.

    We can however achieve a partial inverse of the previous statement if we assume further that $\phi_{1,n}^\#(0)\rightarrow 1$. Specifically, if $\phi_{1,n}^\#(0)\rightarrow 1$ and if for each $K\subset \mathbb{D}$ compact, there is $n_0\in \mathbb{N}$ satisfying that $\phi_{1,n}(K)\cap \phi_{2.n}(K)=\emptyset$, for each $n\geq n_0$, then there is subsequence $(n_k)$ such that $\rho(\phi_{1,n_k}(0),\phi_{2,n_k}(0))\rightarrow 1$. Indeed, for $u_{1,n}$ defined as before, we get by the usual combination of the Schwarz lemma together with a normal families argument, the existence of a subsequence $(n_k)$, and a $\lambda \in \mathbb{T}$, such that 
    $$
    (u_{1,n_k}\circ \phi_{1,n_k})(z)\rightarrow \lambda z
    $$
    uniformly on compact subsets of $\mathbb{D}$. Since the maps $u_{1,n}$ are automorphisms of $\mathbb{D}$, it holds that 
    $$
    (u_{1,n_k}\circ \phi_{1,n_k})(K)\cap (u_{1,n_k}\circ \phi_{2,n_k})(K)=\emptyset 
    $$
    eventually. It easily follows now that
    $$
    |(u_{1,n_k}\circ \phi_{2,n_k})(0)|=\rho(\phi_{1,n_k}(0),\phi_{2,n_k}(0))\rightarrow 1.
    $$

    The preceding discussion, allows us to formulate the following detailed theorem which, in a sense, quantifies the distance between disjoint $\mathcal{B}$-universality and disjoint $H(\mathbb{D})$-universality. 

    \begin{theorem}
        Let $(\phi_{1,n})$ and $(\phi_{2,n})$ be sequences of holomorphic self maps of $\mathbb{D}$ such that for some $a,b \in \mathbb{T}$, $\phi_{1,n}(0)\rightarrow a$ and $\phi_{2,n}(0)\rightarrow b$. Consider the following statements.
        \begin{enumerate}
            \item $(C_{\phi_{1,n}},C_{\phi_{2,n}})$ are disjointly universal on $\mathcal{B}$;

            \item there is a subsequence $(n_k)$ such that $\phi_{1,n_k}^\#(0)\rightarrow 1$, $\phi_{2,n_k}^\#(0)\rightarrow 1$, and\\ $\rho(\phi_{1,n_k}(0),\phi_{2,n_k}(0))\rightarrow 1$;

            \item there is a subsequence $(n_k)$ such that $\phi_{1,n_k}^\#(0)\rightarrow 1$, $\phi_{2,n_k}^\#(0)\rightarrow 1$, and satisfying that for each $K\subset \mathbb{D}$ compact, there is $k_0\in \mathbb{N}$ such that $\phi_{1,n_k}(K)\cap \phi_{2,n_k}(K)=\emptyset$, for $k\geq k_0$;

            \item there is a subsequence $(n_k)$ such that for each $K\subset \mathbb{D}$ compact, there is $k_0\in \mathbb{N}$ satisfying that $\phi_{1,n_k}|_K$ and $\phi_{2,n_k}|_K$ are univalent, and that $\phi_{1,n_k}(K)\cap \phi_{2,n_k}(K)=\emptyset$, for $k\geq k_0$;

            \item $(C_{\phi_{1,n}},C_{\phi_{2,n}})$ are disjointly universal on $H(\mathbb{D})$.
        \end{enumerate}
        Then we have the implications $(1)\Leftrightarrow (2)\Leftrightarrow (3)\Rightarrow (4)\Leftrightarrow (5)$, and $(4)\nRightarrow (3)$ in general.
    \end{theorem}

    \begin{pf}
        The only claim that we need to justify is that $(4)\nRightarrow (3)$. This can be easily seen by considering for $n\in \mathbb{N}$, $\phi_{1,n}(z)=a_nz+1-a_n$, for $0<a_n<1$ and $a_n\rightarrow 0$, and $\phi_{2,n}=-\phi_{1,n}$. We notice that the implication $(4)\Leftrightarrow (5)$ is \cite[Theorem 2.1]{BeMa}.
    \end{pf}

    \section{Concluding Remarks}

    As we mentioned before, the technical part in the proof of Theorem \ref{criterion}, is the existence of a disjointly universal vector of the form $\prod_{i=1}^{\infty}x_i$, for $x_i \in D_0$. In situations where the set
    $$
    \{ \prod_{i=1}^{\infty}x_i \in X: x_i D_0\}
    $$
    is residual, the proof would be an immediate consequence of the Baire category theorem and Proposition \ref{Birkhoff}. While writing the paper, we did not know whether the set of Blashke products is residual in $\mathcal{B}$ and similarly whether the set of singular inner functions is residual in $\mathcal{S}$ under the topology of uniform convergence on compact subsets. During a personal communication with Chalmoukis, we shared the two problems with him, he answered both of them positively, and he posted the proofs on his personal page \cite{Cha}. Therefore, our two main results, Theorem \ref{Blaschke} and Theorem \ref{singular}, admit shorter proofs, through the simpler part of Theorem \ref{criterion}. We chose to keep Theorem \ref{criterion} in its original form, for potential future applications in cases where the Baire category theorem does not apply. 

    As noted in the introduction, all our results concerning disjoint universality of a pair of sequences of maps, extend directly to any number of sequences. More specifically, consider the $N\in \mathbb{N}$ sequences of holomorphic self-maps of $\mathbb{D}$, $(\phi_{1,n})_n, \dots ,(\phi_{N,n})_n$, which satisfy $\phi_{i,n}(0)\rightarrow a_i\in \mathbb{T}$, $1\leq i\leq N$. If there is a subsequence $(n_k)$ such that $\phi_{i,n_k}^\#(0)\rightarrow 1$, $1\leq i\leq N$ and $\rho(\phi_{i,n_k}(0), \phi_{j,n_k}(0))\rightarrow 1$, whenever $i\neq j$, then there exists a Blaschke product $B$ such that the set $\{(B\circ \phi_{1,n},\dots ,B\circ \phi_{N,n}): n\in \mathbb{N}\}$ becomes dense in $\mathcal{B}^N$. Furthermore, by Chalmoukis's result \cite{Cha} and a Baire category argument, we get that the set of such disjointly universal Blaschke products is residual in $\mathcal{B}$. We next examine to what extent we can push this argument.

    The definition of disjoint universality can be extended verbatim to a countable infinite collection of sequences of maps. So, the sequences of continuous self-maps on the topological space $X$, $(T_{N,n})_n$, $N\in \mathbb{N}$, will be called {\it disjointly universal}, provided that there exists $x\in X$, which will be called a {\it disjointly universal element} satisfying that the set
    $$
    \{(T_{1,n}(x), T_{2,n}(x),\dots ): n\in \mathbb{N}\}
    $$
    becomes dense in $X^{\mathbb{N}}$, when the latter is endowed with the product topology. The very definition of product topology makes it clear that disjoint universality of a sequence of maps is just {\it common disjoint universality} of the finite truncations of it. More precisely, an $x\in X$ is disjointly universal for $(T_{1,n})_n, (T_{2,n})_n, \dots$ if and only if it is disjointly universal for $(T_{1,n})_n, \dots , (T_{N,n})_n$, for all $N\in \mathbb{N}$. 

    Let now $(\phi_{1,n})_n, (\phi_{2,n})_n,\dots$ be sequences of holomorphic self maps of $\mathbb{D}$, such that $\phi_{i,n}\rightarrow a_i\in \mathbb{T}$, for $i\in \mathbb{N}$. If there exists a subsequence $(n_k)$ such that $\phi_{i,n_k}^\#(0)\rightarrow 1$, $i\in \mathbb{N}$ and $\rho(\phi_{i,n_k}(0), \phi_{j,n_k}(0))\rightarrow 1$, whenever $i\neq j$, then since the set of disjointly universal Blaschke products for $(C_{\phi_{1,n}})_n, \dots ,(C_{\phi_{N,n}})_n$ for $N\in \mathbb{N}$, is residual in $\mathcal{B}$, we get a Blaschke product which is disjointly universal for the sequences $(C_{\phi_{1,n}})_n, (C_{\phi_{2,n}})_n,\dots$. In particular, if we start with a dense subset $\{a_i: i\in \mathbb{N}\}$ of $\mathbb{T}$, and for each $i\in \mathbb{N}$ we consider a disc automorphism $\phi_i \in Aut(\mathbb{D})$ with attractive fixed point at $a_i$, we get a Blaschke product $B$ such that the set $\{(B\circ \phi_1^n, B\circ \phi_2^n,\dots): n\in \mathbb{N}\}$ is dense in $\mathcal{B}^{\mathbb{N}}$. Evidently, the singular set of such a Blaschke product is the whole circle $\mathbb{T}$. An analogous situation holds for singular inner functions, and $\mathcal{S}$ instead of $\mathcal{B}$.

    Finally, let us observe that this is as far as one can hope to get with increasing the number of sequences. Indeed, if $X$ is a topological space with at least two points and if $\kappa$ is an uncountable cardinal number, then $X^{\kappa}$ is not separable with the product topology, hence no set of the form $\{ (T_{i,n}(x))_i: n\in \mathbb{N}\}$ can be dense in $X^{\kappa}$. 

    \section{Appendix: Proof of Theorem \ref{criterion}}

    For the first claim of the theorem we apply Proposition \ref{Birkhoff}. Let $U, V_1$, and $V_2$ be non-empty open subsets of $X$. We pick $x\in D_0\cap U$, $y\in D_1\cap V_1$ and $z\in D_2\cap V_2$. Setting
    $$
    w_k=xS_{1, n_k}yS_{2, n_k}z,
    $$
    we readily see that
    $$
    w_k \rightarrow x, \quad T_{1, n_k} w_k \rightarrow y, \quad \mbox{and}  \quad T_{2, n_k} w_k \rightarrow z,   
    $$
   as $k \rightarrow \infty$.
    Hence, the assumptions of Proposition \ref{Birkhoff} are satisfied for $k$ sufficiently large.

    We now prove the second claim which constitutes the long and technical part of the proof.
    Assume that the topology of $X$ is induced by the metric $d$, and let $(y_k,z_k)_{k=1}^{\infty}\subset D_1\times D_2$ be a dense sequence in $X\times X$. We will construct a disjointly universal vector for $(T_{1,n})_n, (T_{2,n})_n$  by induction on $k\in \mathbb{N}$.

    For the base of the inductive process, let $\delta_1>0$ be such that
    \begin{equation}\label{eq1}
        d(a,y_1)<\delta_1 \quad \text{and} \quad d(b,e)<\delta_1 \Rightarrow d(ab,e)<\frac{1}{2},
    \end{equation}
    and
    \begin{equation}\label{eq2}
        d(a,e)<\delta_1 \quad \text{and} \quad d(b,z_1)<\delta_1 \Rightarrow d(ab,y_1)<\frac{1}{2}.
    \end{equation}
    Pick $\delta_1'>0$ such that
    \begin{equation}\label{eq3}
        d(a,y_1)<\delta_1' \quad \text{and} \quad d(b,e)<\delta_1' \Rightarrow d(ab,y_1)<\frac{\delta_1}{2},
    \end{equation}
    \begin{equation}\label{eq3'}
        d(a,e)<\delta_1' \quad \text{and} \quad d(b,z_1)<\delta_1' \Rightarrow d(ab,z_1)<\frac{\delta_1}{2},
    \end{equation}
    and
     \begin{equation}\label{eq3''}
        d(a,e)<\delta_1' \quad \text{and} \quad d(b,e)<\delta_1' \Rightarrow d(ab,e)<\frac{\delta_1}{2}.
    \end{equation}
    Using the assumptions, we find $n_1\in \mathbb{N}$, such that
    \begin{equation}\label{eq4}
        d(S_{1, n_1}y_1,e)<\delta_1',
    \end{equation}
    \begin{equation}\label{eq4'}
       d(S_{2, n_1}z_1,e)<\delta_1', 
    \end{equation}
    \begin{equation}\label{eq5}
       d(T_{1, n_1}S_{1, n_1}y_1,y_1)<\delta_1',   
    \end{equation}
    \begin{equation}\label{eq5'}
      d(T_{1, n_1}S_{2, n_1}z_1,e)<\delta_1',       
    \end{equation}
    \begin{equation}\label{eq6}
          d(T_{2, n_1}S_{1, n_1}y_1,e)<\delta_1',
    \end{equation}
    and
    \begin{equation}\label{eq6'}
      d(T_{2, n_1}S_{2, n_1}z_1,z_1)<\delta_1'.  
    \end{equation}

    Notice that (\ref{eq3''}), (\ref{eq4}), and  (\ref{eq4'}) imply that
    \begin{equation}\label{eq7}
        d(S_{1, n_1}y_1S_{2, n_1}z_1,e)<\frac{\delta_1}{2},
    \end{equation}
     (\ref{eq3}), (\ref{eq5}) and (\ref{eq5'}) imply that
    \begin{equation}\label{eq8}
        d(T_{1,n_1}S_{1,n_1}y_1 T_{1,n_1}S_{2,n_1}z_1,y_1)<\frac{\delta_1}{2},
    \end{equation}
   and (\ref{eq3'}), (\ref{eq6}), (\ref{eq6'}) imply that
    \begin{equation}\label{eq9}
        d(T_{2, n_1}S_{1,n_1}y_1 T_{2,n_1}S_{2,n_1}z_1,z_1)<\frac{\delta_1}{2}.
    \end{equation}
    
    Let $0<\delta_1'' < \frac{\delta_1}{2}$ be such that
    \begin{equation}\label{eq10}
        d(a,S_{1,n_1}y_1S_{2,n_1}z_1)<\delta_1'' \Rightarrow d(T_{1,n_1}a,T_{1,n_1}(S_{1,n_1}y_1S_{2,n_1}z_1))<\frac{\delta_1}{2}, 
    \end{equation}
    and
    \begin{equation}\label{eq11}
        d(a,S_{1,n_1}y_1S_{2,n_1}z_1)<\delta_1'' \Rightarrow d(T_{2,n_1}a,T_{2,n_1}(S_{1,n_1}y_1S_{2,n_1}z_1))<\frac{\delta_1}{2}. 
    \end{equation}
    Pick $x_1\in D_0$ such that
    \begin{equation}\label{eq12}
        d(x_1,S_{1, n_1}y_1S_{2, n_1}z_1)<\delta_1''.
    \end{equation}
    The triangle inequality together with (\ref{eq7}) and (\ref{eq12}) give
    $$
        d(x_1,e)<\delta_1,
    $$
    together with (\ref{eq8}), (\ref{eq10}), and (\ref{eq12}) give
    $$
        d(T_{1, n_1}x_1,y_1)<\delta_1,
    $$
    and, finally, together with (\ref{eq9}), (\ref{eq11}), and (\ref{eq12}) give 
    $$
        d(T_{2, n_1}x_1,z_1)<\delta_1.
    $$
   
    Let us now assume that the inductive process has been carried out for the first $k-1$ steps, where $k \geq 2$. We choose $\delta_k>0$ such that
    \begin{equation}\label{eq16}
        d(a,e)<\delta_k, \quad d(b,y_k)<\delta_k \quad \text{and} \quad d(c,e)<\delta_k \Rightarrow d(abc,y_k)<\frac{1}{2^k},
    \end{equation}
    \begin{equation}\label{eq17}
        d(a,e)<\delta_k, \quad d(b,z_k)<\delta_k \quad \text{and} \quad d(c,e)<\delta_k \Rightarrow d(abc,z_k)<\frac{1}{2^k},
    \end{equation}
    \begin{equation}\label{eq18}
    d(a,e)<\delta_k \Rightarrow d(a\prod_{i=1}^{k-1}x_i, \prod_{i=1}^{k-1}x_i)<\frac{1}{2^k},
    \end{equation}
    \begin{equation}\label{eq19}
        d(a,e)<\delta_k \Rightarrow d(T_{1, n_{k-1}}a,e)<\delta_{k-1}, 
    \end{equation}
    \begin{equation}\label{eq20}
        d(a,e)<\delta_{k} \Rightarrow d(T_{2, n_{k-1}}a,e)<\delta_{k-1},
    \end{equation}
    and also such that for all $1\leq j\leq k-2$, we have
    \begin{equation}\label{eq21}
        d(a,e)<\delta_k \quad \text{and} \quad d(\prod_{i=j+1}^{k-1}T_{1, n_j}x_i,e)<\delta_j \Rightarrow d(a\prod_{i=j+1}^{k-1}T_{1, n_j}x_i,e)<\delta_j, 
    \end{equation}
    and
    \begin{equation}\label{eq22}
        d(a,e)<\delta_k \quad \text{and} \quad d(\prod_{i=j+1}^{k-1}T_{2, n_j}x_i,e)<\delta_j \Rightarrow d(a\prod_{i=j+1}^{k-1}T_{2, n_j}x_i,e)<\delta_j.
    \end{equation}

    Let $\delta_k'>0$ be such that
    \begin{equation}\label{eq23}
        d(a,y_k)<\delta_k' \quad \text{and} \quad d(b,e)<\delta_k' \Rightarrow d(ab,y_k)<\frac{\delta_k}{2},
    \end{equation}
    \begin{equation}\label{eq24}
        d(a,e)<\delta_k' \quad \text{and} \quad d(b,z_k)<\delta_k' \Rightarrow d(ab,z_k)<\frac{\delta_k}{2},
    \end{equation}
    \begin{equation}\label{eq25}
        d(a_i,e)<\delta_k', i=1,\dots ,k \Rightarrow d(\prod_{i=1}^{k}a_i,e)<\frac{\delta_k}{2}.
    \end{equation}

    Pick $n_k>n_{k-1}$ such that
     \begin{equation}\label{eq26}
        d(S_{1, n_k}y_k,e)<\delta_k'
    \end{equation}
    \begin{equation}\label{eq27}
        d(S_{2, n_k}z_k,e)<\delta_k'
    \end{equation}
    and, such that for each $1\leq j\leq k-1$, 
    \begin{equation}\label{eq28}
       d(T_{1, n_j}(S_{1, n_k}y_k),e)<\delta_k',
    \end{equation}
    \begin{equation}\label{eq29}
       d(T_{1, n_j}(S_{2, n_k}z_k),e)<\delta_k', 
    \end{equation}
     \begin{equation}\label{eq30}
       d(T_{2, n_j}(S_{1, n_k}y_k),e)<\delta_k', 
    \end{equation}
    \begin{equation}\label{eq31}
       d(T_{2, n_j}(S_{2, n_k}z_k),e)<\delta_k', 
    \end{equation}
    \begin{equation}\label{eq32}
       d(T_{1, n_k}S_{1, n_k}y_k,y_k)<\delta_k',  
    \end{equation}
    \begin{equation}\label{eq33}
        d(T_{1, n_k}S_{2, n_k}z_k,e)<\delta_k',
    \end{equation}
    \begin{equation}\label{eq34}
         d(T_{2, n_k}S_{1, n_k}y_k,e)<\delta_k',
    \end{equation}
    \begin{equation}\label{eq35}
         d(T_{2, n_k}S_{2, n_k}z_k,z_k)<\delta_k',
    \end{equation}
    \begin{equation}\label{eq36}
     d(T_{1, n_k}x_j,e)<\delta_k',  \quad \text{and}   
    \end{equation}
     \begin{equation}\label{eq37}
     d(T_{2, n_k}x_j,e)<\delta_k'.
    \end{equation}


    From (\ref{eq25}), (\ref{eq26}), and (\ref{eq27}), it follows that
    \begin{equation}\label{eq38}
        d(S_{1, n_k}y_kS_{2, n_k}z_k,e)<\frac{\delta_k}{2}.
    \end{equation}
    From (\ref{eq25}), (\ref{eq28}), and (\ref{eq29}), it follows that for each $1\leq j\leq k-1$,
    \begin{equation}\label{eq39}
       d(T_{1, n_j}(S_{1, n_k}y_kS_{2, n_k}z_k),e)<\frac{\delta_k}{2}.
    \end{equation}
     From (\ref{eq25}), (\ref{eq30}), and (\ref{eq31}), it follows that for each $1\leq j\leq k-1$,
     \begin{equation}\label{eq40}
     d(T_{2, n_j}(S_{1, n_k}y_kS_{2, n_k}z_k),e)<\frac{\delta_k}{2}.
    \end{equation}
    From  (\ref{eq23}), (\ref{eq32}), and (\ref{eq33}), it follows that
    \begin{equation}\label{eq41}
    d(T_{1, n_k}S_{1, n_k}y_kT_{1, n_k}S_{2, n_k}z_k,y_k)<\frac{\delta_k}{2}.
    \end{equation}
   From (\ref{eq24}), (\ref{eq34}), and (\ref{eq35}), it follows that
    \begin{equation}\label{eq42}
d(T_{2, n_k}S_{1, n_k}y_kT_{2, n_k}S_{2, n_k}z_k,z_k)<\frac{\delta_k}{2}.
    \end{equation}
    From (\ref{eq25}), and (\ref{eq36}),  it follows that
    \begin{equation}\label{eq43}
        d(T_{1, n_k}(\prod_{i=1}^{k-1}x_i),e)<\delta_k.
    \end{equation}
    From (\ref{eq25}), and (\ref{eq37}), it follows that
    \begin{equation}\label{eq44}
      d(T_{2, n_k}(\prod_{i=1}^{k-1}x_i),e)<\delta_k.
    \end{equation}

    Find $0<\delta_k''<\frac{\delta_k}{2}$ such that for each $1\leq j\leq k$,
    \begin{equation}\label{eq45}
        d(a,S_{1, n_k}y_kS_{2, n_k}z_k)<\delta_k'' \Rightarrow d(T_{1, n_j}a, T_{1, n_j}(S_{1, n_k}y_kS_{2, n_k}z_k))<\frac{\delta_k}{2}
    \end{equation}
    \begin{equation}\label{eq46}
       d(a,S_{1, n_k}y_kS_{2, n_k}z_k)<\delta_k'' \Rightarrow d(T_{2, n_j}a, T_{2, n_j}(S_{1, n_k}y_kS_{2, n_k}z_k))<\frac{\delta_k}{2}.
    \end{equation}

    Let $x_k \in D_0$ be such that
    \begin{equation}\label{eq47}
        d(x_k,S_{1, n_k}y_kS_{2, n_k}z_k)<\delta_k''.
    \end{equation}

    The triangle inequality together with (\ref{eq18}), (\ref{eq38}), and (\ref{eq47}) give that
    \begin{equation}\label{eq48}
        d(\prod_{i=1}^{k}x_i,\prod_{i=1}^{k-1}x_i)<\frac{1}{2^k}.
    \end{equation}
    The triangle inequality together with (\ref{eq41}), (\ref{eq45}), and (\ref{eq47}) give that
    \begin{equation}\label{eq59}
        d(T_{1, n_k}x_k,y_k)<\delta_k.
    \end{equation}
    The triangle inequality together with (\ref{eq42}), (\ref{eq46}), and (\ref{eq47})  give that
    \begin{equation}\label{eq50}
        d(T_{2, n_k}x_k,z_k)<\delta_k.
    \end{equation}
    The previous steps of the inductive construction together with the triangle inequality and  (\ref{eq19}), (\ref{eq21}), (\ref{eq38}), (\ref{eq39}), (\ref{eq45}), and (\ref{eq47}) give that for each $1\leq j\leq k-1$,
    \begin{equation}\label{eq51}
    d(T_{1, n_j}(\prod_{i=j+1}^{k}x_i),e)<\delta_j.   
    \end{equation}
    Similarly, (\ref{eq20}), (\ref{eq22}), (\ref{eq38}), (\ref{eq40}), (\ref{eq46}), and (\ref{eq47}) imply that for each $1\leq j\leq k-1$,
    \begin{equation}\label{eq52}
       d(T_{2, n_j}(\prod_{i=j+1}^{k}x_i),e)<\delta_j.
    \end{equation}
    This completes the $k$-th step of the inductive construction.

    We notice that by (\ref{eq48}), the sequence $(\prod_{i=1}^{k}x_i)_{k=1}^{\infty}$ is Cauchy hence it converges to an element $x\in X$ and we may write
    \begin{equation*}
        x=\prod_{i=1}^{\infty}x_i.
    \end{equation*}
    We claim that $x$ is a disjointly universal vector for $(T_{1,n})_n$, $(T_{2,n})_n$. Indeed, by (\ref{eq16}), (\ref{eq43}), (\ref{eq59}), and (\ref{eq51}) we have that
    \begin{equation*}
        d(T_{1, n_k}x,y_k)=\lim_{N\rightarrow \infty}d(T_{1, n_k}(\prod_{i=1}^{k-1}x_i)T_{1, n_k}x_k T_{1, n_k}(\prod_{i=k+1}^Nx_i),y_k)\leq \frac{1}{2^k}. 
    \end{equation*}
    Similarly, by (\ref{eq17}), (\ref{eq44}), (\ref{eq50}), and (\ref{eq52}) we get that
    \begin{equation*}
        d(T_{2, n_k}x,z_k)=\lim_{N\rightarrow \infty}d(T_{2, n_k}(\prod_{i=1}^{k-1}x_i)T_{2, n_k}x_k T_{2, n_k}(\prod_{i=k+1}^Nx_i),z_k)\leq \frac{1}{2^k} 
    \end{equation*}
    which concludes the proof.

\end{document}